\documentclass[a4paper,11pt,intlimits,oneside]{amsart}

\usepackage{enumerate, color}
\usepackage{amsfonts,amsmath}
\usepackage{esint}

\usepackage{latexsym,amssymb}

\usepackage{mathrsfs}

\allowdisplaybreaks

\newcommand{\comment}[1]{}

\newcommand{\R}{{\mathbb R}}

\newcommand{\f}{\frac}
\newcommand{\vc}{\infty}

\def\X{{\mathbb R^n_+}}

\def\sign{{\mbox{\rm  sign}}\,}

\newcommand{\ackname}{Acknowledgements}
\makeatletter
\if@titlepage
\newenvironment{acknowledgement}{%
	\titlepage
	\null\vfil
	\@beginparpenalty\@lowpenalty
	\begin{center}%
		\bfseries \ackname
		\@endparpenalty\@M
\end{center}}%
{\par\vfil\null\endtitlepage}
\else
\newenvironment{acknowledgement}{%
	\if@twocolumn
	\section*{\ackname}%
	\else
	\begin{center}%
		{\bfseries \ackname\vspace{0em}\vspace{\z@}}%
	\end{center}%
	\quotation
	\fi}
\fi
\makeatother

\begin{document}
	
	\title[Commutators of the Dirichlet Riesz Transform on the Half-Space]{Endpoint Estimates for Commutators of the Dirichlet Riesz Transform on the Half-Space}
	
	\author{The Anh Bui}  
	\address{Department of Mathematics, Macquarie University, NSW 2109, Australia} 
	\email{{\tt the.bui@mq.edu.au}}
	\author{Xuan Thinh Duong}  
	\address{Department of Mathematics, Macquarie University, NSW 2109, Australia} 
	\email{{\tt xuan.duong@mq.edu.au}}
	\author{Xuejing Huo}  
	\address{Department of Mathematics, Macquarie University, NSW 2109, Australia} 
	\email{{\tt xuejing.huo@hdr.mq.edu.au}}
	\author{Luong Dang Ky$^*$} 
	\address{Department of Education, Quy Nhon University, 170 An Duong Vuong, South Quy Nhon, Gia Lai, Vietnam} 
	\email{{\tt luongdangky@qnu.edu.vn}}

	\thanks{$^*$Corresponding author}

	\begin{abstract}
Let $\mathfrak R=\nabla(-\Delta_D)^{-1/2}$ be the Riesz transform associated with the Dirichlet Laplacian on the upper half-space $\mathbb R^n_+$, and let $H^1_r(\mathbb R^n_+)$ denote the Hardy space of restrictions. We establish an endpoint theory for commutators with $\mathfrak R$ on $H^1_r(\mathbb R^n_+)$. For every $b\in \mathrm{BMO}(\mathbb R^n_+)$, the commutator $[b,\mathfrak R]$ extends boundedly from $H^1_r(\mathbb R^n_+)$ to $L^{1,\infty}(\mathbb R^n_+)$. We then prove the sharp strong-endpoint characterization
\[
[b,\mathfrak R]:H^1_r(\mathbb R^n_+)\to H^1_r(\mathbb R^n_+) \quad\Longleftrightarrow\quad b\in \mathrm{BMO}^{\log}(\mathbb R^n_+).
\]
Moreover, the $\mathrm{BMO}^{\log}$ norm is quantitatively equivalent to the sum of the usual $\mathrm{BMO}$ norm and the $H^1_r$ operator norms of the component commutators $[b,\mathfrak R_j]$. The logarithmic condition records the distance to the boundary and is therefore intrinsic to the Dirichlet geometry.
\end{abstract}

	
	\maketitle
	\newtheorem{theorem}{Theorem}[section]
	\newtheorem{lemma}{Lemma}[section]
	\newtheorem{proposition}{Proposition}[section]
	\newtheorem{remark}{Remark}[section]
	\newtheorem{corollary}{Corollary}[section]
	\newtheorem{definition}{Definition}[section]
	\newtheorem{example}{Example}[section]
	\numberwithin{equation}{section}
	\newtheorem*{theoremjj}{Theorem J-J}
	\newtheorem*{theorema}{Theorem A}
	\newtheorem*{theoremb}{Theorem B}
	\newtheorem*{theoremc}{Theorem C}
	\newtheorem*{conjecture}{Conjecture}
	\newtheorem*{open question}{Open question}

	\section{Introduction and statement of the results}

Let $-\Delta_D$ be the Dirichlet Laplacian on the upper half-space
\[
\mathbb R^n_+:=\{x=(x',x^n)\in\mathbb R^n:x^n>0\}.
\]
Its heat semigroup admits the kernel representation
\[
p_t(x,y)=\frac{1}{(4\pi t)^{n/2}}e^{-\frac{|x'-y'|^2}{4t}}
\left(e^{-\frac{|x^n-y^n|^2}{4t}}-e^{-\frac{|x^n+y^n|^2}{4t}}\right),
\qquad x,y\in\mathbb R^n_+,\, t>0.
\]
The second term in this formula is the reflected contribution that enforces the Dirichlet boundary condition. We write
\[
\mathfrak R=(\mathfrak R_1,\ldots,\mathfrak R_n)
=\nabla(-\Delta_D)^{-1/2},
\qquad
\mathfrak R_j=\partial_j(-\Delta_D)^{-1/2},\quad j=1,\ldots,n,
\]
for the corresponding Dirichlet Riesz transform.

For a locally integrable function $b$ and a linear operator $T$, the commutator of $b$ and $T$ is defined by
\[
[b,T]f=bTf-T(bf).
\]
For the vector Riesz transform, we write
\[
[b,\mathfrak R]f:=\big([b,\mathfrak R_1]f,\ldots,[b,\mathfrak R_n]f\big).
\]
Vector-valued $H^1_r$ and $L^{1,\infty}$ norms are understood componentwise, with the sum of the component norms. Since the dimension is fixed, any equivalent finite-dimensional choice leads to the same boundedness statements.
The boundedness of commutators of singular integral operators is a central theme in harmonic analysis. The foundational theorem of Coifman, Rochberg, and Weiss \cite{CRW} states that, whenever $T$ is a Calder\'on--Zygmund operator and $b\in\mathrm{BMO}(\mathbb R^n)$, the commutator $[b,T]$ is bounded on $L^p(\mathbb R^n)$ for every $1<p<\infty$. Thus, BMO is the natural symbol class for the $L^p$ theory of commutators.

The endpoint $p=1$ is more subtle. In general, $[b,T]$ need not map $H^1(\mathbb R^n)$ boundedly into $L^1(\mathbb R^n)$. One classical approach is to associate with each $b\in\mathrm{BMO}$ a maximal subspace $H^1_b\subset H^1(\mathbb R^n)$ on which the commutator has an $L^1$ endpoint bound; see, for instance, \cite{Pe,Ky13}. Since this space depends on the symbol $b$, this approach does not provide a fixed Hardy-space endpoint for the whole BMO class. Endpoint commutator estimates in operator-adapted settings have also been studied for singular integrals related to Schr\"odinger operators; see \cite{Ky15}. We also refer to \cite{Ky25} for generalized Calder\'on--Zygmund operators on the restriction Hardy space.

The purpose of this paper is to study the corresponding endpoint problem for $\mathfrak R$ on $\mathbb R^n_+$. The operator $\mathfrak R$ is a Calder\'on--Zygmund operator in the present setting, so the usual $L^p$ commutator bounds hold for $1<p<\infty$. Our focus is instead the fixed Hardy space of restrictions $H^1_r(\mathbb R^n_+)$, consisting of restrictions to $\mathbb R^n_+$ of functions in $H^1(\mathbb R^n)$. The key point is that the cancellation condition for atoms in $H^1_r(\mathbb R^n_+)$ depends on their position relative to the boundary. This boundary dependence is reflected by the logarithmic BMO space $\mathrm{BMO}^{\log}(\mathbb R^n_+)$ defined in \eqref{eq-bmolog-definition}, whose norm involves the factor
\[
\log\Big(e+\frac{x_B^n}{r_B}\Big).
\]
For balls at the boundary scale this factor is bounded, whereas for balls lying far from the boundary relative to their radius it becomes large. Thus $\mathrm{BMO}^{\log}$ imposes precisely the additional decay of mean oscillation that is needed away from the boundary.

Our main theorem shows that this logarithmic condition is not merely sufficient: it is exactly the symbol condition for strong boundedness of the commutator on $H^1_r(\mathbb R^n_+)$. At the same time, every BMO symbol still gives the natural weak endpoint estimate. In this sense, the theorem separates sharply the weak and strong endpoint theories. The necessity direction is genuinely boundary-sensitive: it combines a Riesz-transform characterization of $H^1_r$ with a new atomic characterization of $\mathrm{BMO}^{\log}$, rather than following solely from the usual Calder\'on--Zygmund commutator theory.

The main result is the following.

\begin{theorem}\label{mainthm}
\begin{enumerate}[{\rm (a)}]
\item If $b\in\mathrm{BMO}(\X)$, then $[b,\mathfrak R]$ extends to a bounded operator from $H^1_r(\R_+^n)$ into $L^{1,\vc}(\R_+^n)$.

\item Let $b\in\mathrm{BMO}(\X)$. Then $[b,\mathfrak R]$ is bounded on $H^1_r(\X)$ if and only if $b\in\mathrm{BMO}^{\log}(\X)$. Moreover, there is a constant $C>1$, independent of $b$, such that
\[
C^{-1}\|b\|_{\mathrm{BMO}^{\log}(\X)}
\leq \|b\|_{\mathrm{BMO}(\X)}+\sum_{j=1}^{n}\|[b,\mathfrak R_j]\|_{H^1_r\to H^1_r}
\leq C\|b\|_{\mathrm{BMO}^{\log}(\X)}.
\]
\end{enumerate}
\end{theorem}

Part~(a) gives the endpoint weak-type estimate for the full BMO class. Part~(b) shows that strong boundedness on the fixed space $H^1_r(\mathbb R^n_+)$ requires, and is equivalent to, the additional boundary-sensitive logarithmic regularity of the symbol. In particular, the logarithmic condition is necessary and not an artifact of the proof. The norm equivalence is quantitative: the $\mathrm{BMO}^{\log}$ norm is controlled by, and controls, the usual BMO norm together with the $H^1_r$ operator norms of the component commutators.

Two ingredients are central to the proof. First, the reflected structure of the Dirichlet heat kernel yields off-diagonal estimates for $\mathfrak R$ acting on restriction-Hardy atoms, with an additional boundary decay in the non-cancellative case. Second, we introduce logarithmic molecules for $H^1_r(\mathbb R^n_+)$ and prove an atomic characterization of $\mathrm{BMO}^{\log}(\mathbb R^n_+)$ tailored to the product $a(b-b_B)$. These two mechanisms allow the boundary geometry to enter explicitly into both the sufficiency and the necessity arguments.

The remainder of the paper is organized as follows. Section~2 develops the required background. We recall the atomic description of $H^1_r(\mathbb R^n_+)$ in Section~2.1, introduce the relevant BMO-type spaces and prove the characterization of $\mathrm{BMO}^{\log}(\mathbb R^n_+)$ in Section~2.2, and establish the needed estimates for the Dirichlet Riesz transform in Section~2.3. Section~3 contains the proofs of the main theorem. The proof of Lemma~\ref{lem-gx0} is given in the appendix.

\bigskip

\noindent\textbf{Notation.}
\begin{itemize}
\item The letter $C$ denotes a positive constant, independent of the principal parameters under consideration, whose value may change from line to line. We write $f\lesssim g$ if $f\leq Cg$, and $f\simeq g$ if $f\lesssim g\lesssim f$.
\item If $B=B(x,r)$ is a ball and $t>0$, then $tB:=B(x,tr)$. For a set $E\subset\X$, let $\chi_E$ denote its characteristic function.
\item For $q\in[1,\infty]$, $q'$ denotes the conjugate exponent. We set $\mathbb Z^+:=\{1,2,\ldots\}$ and $\mathbb Z_0^+:=\mathbb Z^+\cup\{0\}$.
\item For a ball $B\subset\R_+^n$, $x_B$ and $r_B$ denote its centre and radius, respectively. For $x\in\R_+^n$, we write $x=(x',x^n)\in\R^{n-1}\times(0,\infty)$.
\item For every ball $B$, set $S_0(B):=B$ and $S_j(B):=2^jB\setminus2^{j-1}B$ for $j\in\mathbb Z^+$.
\item If $E\subset\R_+^n$ is measurable, then
\[
\fint_E f(x)\,dx:=\frac{1}{|E|}\int_E f(x)\,dx.
\]
\end{itemize}

\section{Preliminaries}\label{sec-hardy-restriction}
	
	 \subsection{The Hardy space of restriction $H^1_r(\R_+^n)$} 
	 	We first recall the atomic decomposition of the Hardy space of restriction. See, for example, \cite{Mi} and \cite[Section 4]{BDK}. 
	 
	 \begin{definition} 
	 	Let $1<q\le \infty$. A bounded, measurable function $a: \X\to \mathbb{R}$ is called an $(H^1_{r},q)$-atom if 
	 	\begin{enumerate}[\upshape (i)]
	 		\item $a$ is supported in a ball $B\subset \X$;
	 		\item $\|a\|_{L^q(\X)}\le |B|^{1/q-1}$;
	 		\item either $x_B^n/4<r_B<x_B^n/2$, or $r_B\le x_B^n/4$ and 
	 		\[
	 		\int a(x)\,dx=0.
	 		\]
	 	\end{enumerate}
	 \end{definition}
	 The Hardy space $H^{1,q}_{r}(\X)$ is defined as the set of all $f\in L^1(\X)$ such that
	 \[
	 f=\sum_j\lambda_j a_j
	 \]
	 where $a_j$ are $(H^1_{r},q)$-atoms and $\lambda_j$ are scalars satisfying $\sum_{j}|\lambda_j|<\infty$.
	 We also set
	 \[
	 \|f\|_{H^{1,q}_{r}(\X)}=\inf\Big\{\sum_{j}|\lambda_j|: f=\sum_j\lambda_j a_j \Big\},
	 \]
	 where the infimum is taken over all such decompositions.
	 
	 \begin{remark}
	 	It was proved that
	 	\[
	 	H^{1,q}_r(\mathbb R^n_+)=H^{1,\infty}_r(\mathbb R^n_+)
	 	\quad \text{for all } q\in(1,\infty].
	 	\]
	 	Therefore, the Hardy space $H^1_r(\mathbb R^n_+)$ may be defined as any $H^{1,q}_r(\mathbb R^n_+)$ with $q\in(1,\infty]$. See \cite{YYZ} (see also \cite[Section 4]{BDK}).
	 \end{remark}

	 It is known that the atomic space $H^1_r(\X)$  above coincides with  the Hardy space of restriction $H^1_{\mathrm{res}}(\X)$ defined by
	 \[
	 H^1_{\mathrm{res}}(\X)
	 =
	 \big\{ f\in L^1(\X) : \exists\, F\in H^1(\mathbb{R}^n)
	 \text{ such that } F|_{\X}=f \big\},
	 \]
	 and
	 \[
	 \|f\|_{H^1_{\mathrm{res}}(\X)}
	 =
	 \inf \big\{ \|F\|_{H^1(\mathbb{R}^n)} :
	 F\in H^1(\mathbb{R}^n),\ F|_{\X}=f \big\}.
	 \]
	 See \cite[Remark 4.5]{BDK}.

	We now introduce a logarithmic molecular notion adapted to the Hardy space $H^1_r(\R_+^n)$.
	\begin{definition}\label{def-log-molecule} Let $s\in (1,\vc)$. A measurable function $m$ is called a $(s,\varepsilon)$-molecule of log-type associated with a ball $ B=B(x_B,r_B)$ if 
			\begin{enumerate}[\rm (a)]
				\item $\|m\|_{L^s(S_j(B))}\leq 2^{-j\varepsilon} |2^j B|^{1/s-1}$ for all $j\in\mathbb{Z}^+_0$,
				
				\item  $\displaystyle\Big|\int_{\X}m(x)dx\Big|\leq  \frac{1}{\log\Big(e+ \frac{x_B^n}{r_B}\Big)}$.
			\end{enumerate}
	 A $(s,\varepsilon)$-molecule of log-type associated with a ball $B$ is called an $s$-atom of log-type if $\operatorname{supp}m\subset B$.
	\end{definition}

The following molecular estimate will be used below:
	\begin{lemma}\label{lem-log-molecule}
		Let  $q\in (1,\vc)$ and $\varepsilon>0$. Then there exists a constant $C>0$ such that, for all $(q, \varepsilon)$-molecules of log-type  $m$,
		\[\|m\|_{H^1_{r}(\mathbb R^n_+)}\leq C.\]
	\end{lemma}

We begin the proof of Lemma~\ref{lem-log-molecule} with two auxiliary facts.	
	\begin{definition}\label{def-bmoz}
		
		A function $b\in L^1_{\mathrm{loc}}(\mathbb{R}^n_+)$ is said to belong to 
		$\mathrm{BMO}_z(\mathbb{R}^n_+)$ if
		\[
		\|b\|_{\mathrm{BMO}_z(\mathbb{R}^n_+)}:=\sup_{\substack{B:\,\text{ball}\\ r_B<x_B^n/4}}\frac{1}{|B|}\int_B |b(x)-b_B|\,dx +
		\sup_{\substack{B:\,\text{ball}\\ r_B\ge x_B^n/4}}\frac{1}{|B|}\int_B |b(x)|\,dx <\vc.
		\]
	\end{definition}

Obviously,
		\[
		\|f\|_{\mathrm{BMO}(\mathbb{R}^n_+)}\lesssim \|f\|_{\mathrm{BMO}_z(\mathbb{R}^n_+)}.
		\]
		In addition, the space $\mathrm{BMO}_z(\mathbb{R}^n_+)$ is the dual space of 
		$H^1_r(\mathbb{R}^n_+)$, that is,
		\[
		\big(H^1_r(\mathbb{R}^n_+)\big)^*
		=
		\mathrm{BMO}_z(\mathbb{R}^n_+).
		\]
		See, for example, \cite[Theorem 2.1]{YYZ}. The notation $\mathrm{BMO}_z$ reflects the fact that the space in Definition~\ref{def-bmoz} coincides, with equivalent norm, with the BMO space obtained by zero extension. More precisely, it consists of all functions $f$ whose zero extension $f_z$ to $\mathbb{R}^n$ belongs to $\mathrm{BMO}(\mathbb{R}^n)$, with norm
		\[
		\|f\|_{\mathrm{BMO}_z(\mathbb{R}^n_+)} \simeq \|f_z\|_{\mathrm{BMO}(\mathbb{R}^n)}.
		\]

	We also use the following duality result, taken from \cite[Theorem 4.4]{CY}. 
	
	\begin{proposition} 
		\label{prop-vmoz-duality} Let ${\rm VMO}_z(\R^n_+)$ be the closure of $C^\vc_c(\R_+^n)$ in ${\rm BMO}_z(\R^n_+)$. Then we have
		\[
		({\rm VMO}_z(\R^n_+))^* = H^1_r(\R_+^n).
		\]
	\end{proposition}

	\begin{proof}
		We set
		for $j\ge 0$, $\displaystyle\alpha_j = \int_{S_j(B)} m(x)dx$ and $\chi_j=\f{1}{|S_j(B)|}1_{S_j(B)}$. Then we define
		$$
		a_j(x)=m(x)\chi_{S_j(B)}(x)-\alpha_j \chi_j(x).
		$$
		If we set $N_j=\sum_{k=j}^\vc \alpha_k$, then we have
		\begin{equation}\label{eq-molecule-decomposition}
			\begin{aligned}
				m(x)&=\sum_{j=0}^\vc a_j(x) +\sum_{j=0}^\vc N_{j+1}(\chi_{j+1}(x)-\chi_j(x))+ \chi_0(x)\int m(y)dy\\
				&=\sum_{j=0}^\vc a_j(x)+ \sum_{j=0}^\vc b_j(x) + a(x).
			\end{aligned}
		\end{equation}
	For each $j$, we have 
		\begin{equation}\label{eq-aj-atom-bounds}
			{\rm supp}\,a_j\subset 2^jB, \int a_j =0 \quad \text{and} \quad \|a_j\|_{L^q}\leq C 2^{-j\varepsilon}|2^jB|^{1/q-1}.
		\end{equation}
By \eqref{eq-aj-atom-bounds}, the zero extension of $2^{j\varepsilon}a_j$ to $\mathbb R^n$ is, up to a uniform constant, a classical $(H^1,q)$-atom supported in a ball comparable with $2^jB$. Hence
\[
\|a_j\|_{H^1_r(\X)}\lesssim 2^{-j\varepsilon},
\]
and therefore
\[
\Big\|\sum_{j}a_j\Big\|_{H^1_r(\X)}\lesssim 1.
\]

		Next we also observe that
		\begin{equation}\label{eq-bj-support}
			{\rm supp}\,b_j\subset 2^{j+1}B \quad \text{and} \quad \int b_j =0.
		\end{equation}
		Moreover,
		$$
		\|b_j\|_{L^q}\leq |N_{j+1}||2^jB|^{1/q-1}.
		$$
		From the definition of $N_{j+1}$ and H\"older's inequality, we obtain
		$$
		\begin{aligned}
			|N_{j+1}|&\leq \sum_{k\geq j+1}\int_{S_k(B)}|m(y)|dy\leq  \sum_{k\geq j+1}|2^kB|^{1-1/q}\|m\|_{L^q(S_k(B))}\\
			&\leq \sum_{k\geq j+1}2^{-k\varepsilon}\\
			& \leq C2^{-j\varepsilon}.
		\end{aligned}
		$$
		This implies that
		\begin{equation}\label{eq-bj-size}
			\|b_j\|_{L^q}\leq C2^{-j\varepsilon}|2^jB|^{1/q-1}.
		\end{equation}
		Similarly, since each $b_j$ has mean zero, \eqref{eq-bj-support}--\eqref{eq-bj-size} and the same zero-extension argument give
\[
\|b_j\|_{H^1_r(\X)}\lesssim 2^{-j\varepsilon},
\qquad
\Big\|\sum_{j}b_j\Big\|_{H^1_r(\X)}\lesssim 1.
\]

It remains to prove that $\|a\|_{H^1_r}\lesssim 1$.  By Proposition \ref{prop-vmoz-duality} we claim that
		$$
		\Big|\int_B a(x)\phi(x)dx\Big|\leq C\|\phi\|_{BMO_z},
		$$
		for all $\phi\in C_c^\vc(\X)$.
		
		Indeed, we have
		$$
		\Big|\int_B a(x)\phi(x)dx\Big|\leq \Big|\int_B a(x)(\phi(x)-\phi_B)dx\Big| + |\phi_B|\Big|\int_B a(x)dx\Big|.
		$$
		By H\"older's inequality we have
		$$
		\begin{aligned}
			\Big|\int_B a(x)(\phi(x)-\phi_B)dx\Big|&\leq \|a\|_{L^q(B)}\Big(\int_B|\phi(x)-\phi_B|^{q'}dx\Big)^{1/q'}\\
			&\leq C |B|^{1/q-1}|B|^{1/q'}\|\phi\|_{BMO}\\
			&\leq C\|\phi\|_{BMO_z}.
		\end{aligned}
		$$
		To dominate the second term, by the standard argument, we have 
		$$
		|\phi_B|\lesssim \|\phi\|_{BMO_z}\log\Big(e+\f{x^n_B}{r_B}\Big).
		$$
		Inserting this estimate into the second term gives
		$$
		\begin{aligned}
			|\phi_B|\Big|\int_B a(x)dx\Big|&\lesssim \|\phi\|_{BMO_z}\log\Big(e+\f{x^n_B}{r_B}\Big) \Big|\int_B m(x)dx\Big| \lesssim \|\phi\|_{BMO_z}.
		\end{aligned}
		$$
		This completes the proof.
	\end{proof}
	
\subsection{$\mathrm{BMO}$-type spaces and a characterization of $\mathrm{BMO}^{\log}$}
In this section, we introduce $\mathrm{BMO}^{\log}$ and establish a characterization that will be used in the proof of the main theorem.

Here and in what follows, for any ball $B\subset \X$ and $f\in L^1_{\rm loc}(\X)$, we denote 
\begin{equation}
	f_B:= \frac{1}{|B|}\int_B f(y)dy\quad\text{and}\quad \mathrm{MO}(f,B):= \frac{1}{|B|}\int_B |f(x)- f_B|dx.
\end{equation}	

A locally integrable function $f$ belongs to $\mathrm{BMO} (\X)$ if
\begin{equation}
	\|f\|_{\mathrm{BMO}(\X)}=\sup_{B}   \mathrm{MO}(f,B) <\infty,
\end{equation}
where the supremum is taken over all balls $B=B(x_B,r_B)\subset\X$. 

We recall the following two fundamental properties of BMO functions: 
\begin{itemize}
	\item for every $1<p<\infty$ there exists a constant 
	$C_p>0$ such that for every ball $B\subset \X$,
	\begin{equation}\label{eq-classical-JN}
		\Big( \frac{1}{|B|}\int_B |f(x)-f_B|^p \, dx \Big)^{1/p}
		\le C_p \|f\|_{\mathrm{BMO}(\X)};
	\end{equation}
	\item for any ball $B\subset \X$ and any integer $j\ge 0$, one has
	\begin{equation}\label{eq-BMO-dilated-averages}
		|f_{2^j B}-f_B|
		\lesssim (j+1)\,\|f\|_{\mathrm{BMO}(\X)},
	\end{equation}
	where the implicit constant is independent of $f$, $B$, and $j$.
	
\end{itemize}

A locally integrable function $f$ belongs to $\mathrm{BMO}^{\log}(\X)$ if
\begin{equation}\label{eq-bmolog-definition}
	\|f\|_{\mathrm{BMO}^{\log}(\X)}=\sup_{B}  \log\Big(e+\frac{x_B^n}{r_B} \Big)\mathrm{MO}(f,B)<\infty,
\end{equation}
where the supremum is taken over all balls $B=B(x_B,r_B)\subset\X$. 	
\begin{remark}\label{rem-bmolog-embeds-bmo}
	For  every $f\in \mathrm{BMO}^{\log}(\X)$, we have
	\[\|f\|_{\mathrm{BMO}(\X)}\leq \|f\|_{\mathrm{BMO}^{\log}(\X)}.\]
\end{remark}

	The next lemma is a logarithmic analogue of the John--Nirenberg estimate for classical BMO functions.
\begin{lemma}\label{lem-bmolog-JN}
	Let $p\in [1,\vc)$. Then 
	\begin{equation}\label{eq-bmolog-JN}
		\log\Big(e +\frac{x_B^n}{r_B}\Big)\Big(\fint_{B} |f(x)- f_{B}|^p dx\Big)^{1/p}\leq  C\|f\|_{\mathrm{BMO}^{\log}}
	\end{equation}
	for all balls $B= B(x_B,r_B)$.
\end{lemma}
\begin{proof}
	 
	Define $h:\X\to [0,1]$ by
	\begin{equation*}
		h(x)=
		\begin{cases}
			1, & x\in B,\\
			\frac{2 r_B- |x-x_B|}{r_B},  & x\in 2 B\setminus B,\\
			0, & x\notin 2 B.
		\end{cases}
	\end{equation*}
	Clearly, 
	\begin{equation}\label{eq-cutoff-Lipschitz}
		|h(x)- h(y)|\le \frac{|x-y|}{r_B}, \ \ \ x,y\in\X.
	\end{equation}

Define $\widetilde f^{2B}:= f- f_{2 B}$. Since $h(x)=1$ on $B$, using \eqref{eq-classical-JN} we have
	\begin{align*}
		\Big( \fint_{B} |f(x)- f_{B}|^p dx\Big)^{1/p} &= \Big( \fint_{B} |h(x)\widetilde f^{2B}(x)- (h\widetilde f^{2B})_{B}|^p dx\Big)^{1/p}\\
		& \lesssim \|h\widetilde f^{2B}\|_{\mathrm{BMO}(\X)}.
	\end{align*}

	Thus, it suffices to prove that 
	$$\|h\widetilde f^{2B}\|_{\mathrm{BMO}(\X)}\lesssim \frac{1}{\log\Big(e + \frac{x_B^n}{r_B}\Big)}\|f\|_{\mathrm{BMO}^{\log}},$$
	or equivalently,
	\begin{equation}\label{eq-localized-bmolog-bound}
		\mathrm{MO}(h\widetilde f^{2B}, \widetilde{B}) \lesssim \frac{1}{\log\Big(e + \frac{x_B^n}{r_B}\Big)}\|f\|_{\mathrm{BMO}^{\log}}\quad\text{for all balls } \widetilde{B}= B(x_{\widetilde{B}},r_{\widetilde{B}}).
	\end{equation}
	
	Since $h(x)=0$ for all $x\notin 2B$, we only need to consider balls $\widetilde{B}$ satisfying $\widetilde{B}\cap 2B\ne\emptyset$. We distinguish two cases.
	
	\medskip
	
	{\sl Case 1:} $r_{ \widetilde{B}}> r_B$. Then, $\widetilde{B}\cap (2 B)\ne\emptyset$ implies that $2 B\subset 5 \widetilde{B}$. Therefore, using  the fact that $h(x)=0$ for all $x\notin 2 B$, we obtain
	\begin{align*}
		\mathrm{MO}(h\widetilde f^{2B}, \widetilde{B}) &\leq  2 \fint_{\widetilde{B}} |h(x) \widetilde f^{2B}(x)|dx\\
		&\lesssim \fint_{5 \widetilde{B}} |h(y) \widetilde f^{2B}(y)|dy =\fint_{5 \widetilde{B}} |h(y)[f(y)-f_{2B}]|dy\\
		&\lesssim  \frac{1}{|2 B|}\int_{2 B}|f(x)- f_{2 B}|dx\\
		&\lesssim \frac{1}{\log\Big(e +\frac{x_B^n}{2 r_B}\Big)}\|f\|_{\mathrm{BMO}^{\log}} \simeq \frac{1}{\log\Big(e +\frac{x_B^n}{ r_B}\Big)}\|f\|_{\mathrm{BMO}^{\log}}.
	\end{align*}
	
	{\sl Case 2:} $r_B \geq r_{ \widetilde{B}}$. We first prove that 
	\begin{equation}\label{eq-boundary-log-comparison}
		\log\Big(e+ \frac{x^n_B}{r_B}\Big) \lesssim  \log\Big(e+ \frac{x^n_{\widetilde B}}{r_{B}}\Big) \ \text{and}  \ \log\Big(e+ \frac{x^n_B}{r_B}\Big) \lesssim  \log\Big(e+ \frac{x^n_{\widetilde B}}{r_{\widetilde B}}\Big).
	\end{equation}
	Since the second inequality is a consequence of the first one, we need only to prove the first inequality. To do this, from the facts $\widetilde B\cap 2B\ne \emptyset$ and $r_B>r_{\widetilde B}$, we have
	\[
	3r_B \ge  |x_B-x_{\widetilde B}|\ge |x^n_B-x^n_{\widetilde B}|,
	\] 
	which implies
	\[
	x^n_B\le 3r_B +x^n_{\widetilde B}.
	\]
	As a consequence,
	\[
	1+ \frac{x^n_B}{r_B}\le 1+\f{3r_B +x^n_{\widetilde B}}{r_B}\simeq 1 +\f{x^n_{\widetilde B}}{r_B},
	\]
	which implies \eqref{eq-boundary-log-comparison}.
	
	We now prove \eqref{eq-localized-bmolog-bound}. We first write
		\begin{align*}
		\mathrm{MO}(h \widetilde f^{2B}, \widetilde{B})	&\leq 2  \fint_{\widetilde{B}} |h(x) \widetilde f^{2B}(x)- h_{\widetilde{B}} (\widetilde f^{2B})_{\widetilde{B}}|dx \\
		&\leq 2  \fint_{\widetilde{B}} |h(x)(\widetilde f^{2B}(x)- (\widetilde f^{2B})_{\widetilde{B}})| dx+ 2 |(\widetilde f^{2B})_{\widetilde{B}}|  \fint_{\widetilde{B}} \fint_{\widetilde{B}} |h(y)- h(x)|dxdy\\
		&\leq 2\fint_{\widetilde{B}} |f(x)- f_{\widetilde{B}}|dx +   \frac{2r_{ \widetilde{B}}}{r_B} |(\widetilde f^{2B})_{\widetilde{B}}| \\
		&\leq 2\fint_{\widetilde{B}} |f(x)- f_{\widetilde{B}}|dx +   \frac{2r_{ \widetilde{B}}}{r_B}  |f_{\widetilde{B}} - f_{2 B}|,
	\end{align*}
	where in the third inequality we used \eqref{eq-cutoff-Lipschitz} and $0\le h\le 1$.
	
Using \eqref{eq-boundary-log-comparison},
\[
\begin{aligned}
	\fint_{\widetilde{B}} |f(x)- f_{\widetilde{B}}|dx &\leq \frac{1}{\log \Big(e+ \frac{x^n_{\widetilde{B}}}{r_{ \widetilde{B}}}\Big)}\|f\|_{\mathrm{BMO}^{\log}}\\
	& \lesssim\frac{1}{\log \Big(e+ \frac{x^n_{{B}}}{r_{ {B}}}\Big)}\|f\|_{\mathrm{BMO}^{\log}}.
\end{aligned}
\]	
	It remains to show that 
	\begin{equation}
		\label{eq-average-difference-bound}
		\frac{r_{ \widetilde{B}}}{r_B}  |f_{\widetilde{B}} - f_{2 B}|\lesssim \frac{1}{\log \Big(e+ \frac{x^n_{{B}}}{r_{ {B}}}\Big)}\|f\|_{\mathrm{BMO}^{\log}}. 
	\end{equation}
	To do this, we write
	\[
	\begin{aligned}
		\frac{r_{ \widetilde{B}}}{r_B}  |f_{\widetilde{B}} - f_{2 B}| \le \frac{r_{ \widetilde{B}}}{r_B}  |f_{B(x_{\widetilde B}, r_{\widetilde B})} - f_{B(x_{\widetilde B}, 5r_{ B})}|+\frac{r_{ \widetilde{B}}}{r_B}  | f_{B(x_{\widetilde B}, 5r_{ B})}-f_{B(x_{B}, 2r_{ B})}|=:E+F.
	\end{aligned}
	\]
	Using the standard argument, 
	\[
	E \lesssim \frac{r_{ \widetilde{B}}}{r_B} \sum_{1\le j\le \lceil\log_2(5r_B/r_{\widetilde B})\rceil}   \frac{1}{\log \Big(e+ \frac{x^n_{\widetilde{B}}}{2^jr_{ {\widetilde B}}}\Big)}\|f\|_{\mathrm{BMO}^{\log}}.
	\]
	Since $2^jr_{ {\widetilde B}}\lesssim r_B$ for $1\le j\le \lceil\log_2(5r_B/r_{\widetilde B})\rceil$, we further obtain
	\[
	\begin{aligned}
		E&\lesssim \frac{r_{ \widetilde{B}}}{r_B} \sum_{1\le j\le \lceil\log_2(5r_B/r_{\widetilde B})\rceil}   \frac{1}{\log \Big(e+ \frac{x^n_{\widetilde{B}}}{r_B}\Big)}\|f\|_{\mathrm{BMO}^{\log}}\\
		&\lesssim \frac{r_{ \widetilde{B}}}{r_B}  \log_2\Big(\f{5r_B}{r_{\widetilde B}}\Big)    \frac{1}{\log \Big(e+ \frac{x^n_{\widetilde{B}}}{r_{ {B}}}\Big)}\|f\|_{\mathrm{BMO}^{\log}}\\
		&\lesssim \frac{1}{\log \Big(e+ \frac{x^n_{{B}}}{r_{ {B}}}\Big)}\|f\|_{\mathrm{BMO}^{\log}},
	\end{aligned}
	\]
	where in the last inequality we used \eqref{eq-boundary-log-comparison} and $\f{\log x}{x} \lesssim 1$ for $x\ge 1$.

	For the last term $F$, since $\widetilde B\cap 2B\ne \emptyset$ and $r_B>r_{\widetilde{B}}$, we have $2B\subset B(x_{\widetilde{B}},5r_B)$. Therefore,
	\[
	\begin{aligned}
		F&\le \frac{r_{ \widetilde{B}}}{r_B} \fint_{2B}|f(x)-f_{B(x_{\widetilde{B}},5r_B)}|\,dx\\
		&\lesssim \frac{r_{ \widetilde{B}}}{r_B} \fint_{B(x_{\widetilde{B}},5r_B)}|f(x)-f_{B(x_{\widetilde{B}},5r_B)}|\,dx\\
		&\lesssim \frac{1}{\log \Big(e+ \frac{x^n_{\widetilde{B}}}{r_{ {B}}}\Big)}\|f\|_{\mathrm{BMO}^{\log}}\\
		&\lesssim \frac{1}{\log \Big(e+ \frac{x^n_{{B}}}{r_{ {B}}}\Big)}\|f\|_{\mathrm{BMO}^{\log}},
		\end{aligned}
		\]
		where in the last inequality we used \eqref{eq-boundary-log-comparison}.
		
		This completes the proof.
\end{proof}

Using Lemma~\ref{lem-bmolog-JN}, we obtain the following estimate:
	\begin{lemma}\label{lem-bmolog-dilated-ball}
		Let $p\in [1,\infty)$. Then there exists a constant $C>0$ such that
		$$\Big( \fint_{2^j B} |f(x)- f_{B}|^p dx\Big)^{1/p}\leq C \frac{(j+1)}{\log\Big(e +\frac{x_B^n}{2^j r_B}\Big)}\|f\|_{\mathrm{BMO}^{\log}}$$
		for all $f\in \mathrm{BMO}^{\log}(\X)$, all balls $B$, and all $j\ge 0$. 
	\end{lemma}
	
	\begin{proof}
		We write
		\[
		\begin{aligned}
			\Big( \fint_{2^j B} |f(x)- f_{B}|^p \,dx\Big)^{1/p} 
			&\le \Big( \fint_{2^j B} |f(y)- f_{2^j B}|^p \,dy\Big)^{1/p} + \sum_{k=0}^{j-1} \Big|f_{2^{k+1}B}- f_{2^k B}\Big| \\
			&\le \Big( \fint_{2^j B} |f(y)- f_{2^j B}|^p \,dy\Big)^{1/p} + \sum_{k=0}^{j-1} \fint_{2^k B} \Big|f - f_{2^{k+1}B}\Big| \,dx \\
			&\lesssim \Big( \fint_{2^j B} |f(y)- f_{2^j B}|^p \,dy\Big)^{1/p} + \sum_{k=0}^{j-1} \fint_{2^{k+1} B} |f - f_{2^{k+1}B}| \,dx.
		\end{aligned}
		\]
		
		Clearly,
		\[
		\begin{aligned}
			\sum_{k=0}^{j-1} \fint_{2^{k+1}B} |f - f_{2^{k+1}B}| \,dx 
			&\lesssim \sum_{k=0}^{j-1} \frac{1}{\log\Big(e + \frac{x_B^n}{2^k r_B}\Big)} \|f\|_{\mathrm{BMO}^{\log}} \\
			&\lesssim \frac{j}{\log\Big(e + \frac{x_B^n}{2^j r_B}\Big)} \|f\|_{\mathrm{BMO}^{\log}}.
		\end{aligned}
		\]
		
		Moreover, applying Lemma \ref{lem-bmolog-JN}, we have
		\[
		\Big( \fint_{2^j B} |f(y)- f_{2^j B}|^p \,dy\Big)^{1/p} \lesssim \frac{1}{\log\Big(e + \frac{x_B^n}{2^j r_B}\Big)} \|f\|_{\mathrm{BMO}^{\log}}.
		\]
		
		Combining these two estimates completes the proof.
	\end{proof}

	The following lemma gives a characterization of $\mathrm{BMO}^{\log}$ that plays a crucial role in the proof of the main result.
	\begin{lemma}\label{lem-bmolog-characterization}
		Let $q\in (1,\infty]$  and $f\in \mathrm{BMO}(\X)$. Then, $f\in \mathrm{BMO}^{\log}(\X)$ if and only if $a(f-f_{B})\in H^1_r(\X)$ for all $(H^1_r,q)$-atoms $a$ associated with balls $B$. Moreover,
		$$\|f\|_{\mathrm{BMO}^{\log}(\X)}\simeq \|f\|_{\mathrm{BMO}(\X)}+ \sup_{\text{\rm $(H^1_r,q)$-atoms $a$ associated with $B$}} \|a(f-f_{B})\|_{H^1_r},$$
		where the supremum is taken over all $(H^1_r,q)$-atoms $a$.		
	\end{lemma}
	We first record the following technical ingredient. For any $x_0\in\X$, we define $g_{x_0}:\X\to \R$ by
	\begin{equation}\label{eq-gx0-definition}
		g_{x_0}(x)=\max\Big\{0, 1+ \log\frac{x_0^n}{|x-x_0|}\Big\}.
	\end{equation}
	Then, we have:
	
	\begin{lemma}\label{lem-gx0}
		There exists a constant $C>0$ such that 
		$$\|g_{x_{0}}\|_{\mathrm{BMO}_z}\leq C\quad\text{for all } x_0\in\X.$$
		Moreover, if $B$ is a ball with $0<r_B<x_B^n$, then $\log\Big(e+\frac{x_B^n}{r_B}\Big)\leq 2\, g_{x_B}(x)$ for all $x\in B$.
	\end{lemma}
	We will give the proof of Lemma \ref{lem-gx0} in the appendix section. We are now ready to give the proof of Lemma \ref{lem-bmolog-characterization}.

	\begin{proof}[Proof of Lemma \ref{lem-bmolog-characterization}]
		Suppose that $f\in \mathrm{BMO}^{\log}(\X)$. By definition,
\[
\|f\|_{\mathrm{BMO}(\X)}\le \|f\|_{\mathrm{BMO}^{\log}(\X)}.
\]
		
		Choose $s>1$ so that $s<q$ when $q<\infty$ (and choose any $s>1$ when $q=\infty$), and let $r\in(1,\infty)$ be determined by
\[
\frac1s=\frac1q+\frac1r,
\]
with the convention $1/q=0$ when $q=\infty$. Then, for any $(H^1_r,q)$-atom $a$ associated with a ball $B$, H\"older's inequality and the John--Nirenberg inequality give
		\begin{equation}\label{eq-product-Ls}
			\begin{aligned}
				\|a(f-f_{B})\|_{L^{s}}&\le \|a\|_{L^q}\|f-f_{B}\|_{L^{r}(B)}\\
				&\lesssim |B|^{\frac1q-1}|B|^{\frac1r}\|f\|_{\mathrm{BMO}(\X)}\\
				&\lesssim \|f\|_{\mathrm{BMO}^{\log}(\X)}|B|^{\frac{1}{s}-1}.
			\end{aligned}
		\end{equation}
		Moreover, by H\"older's inequality and Lemma \ref{lem-bmolog-dilated-ball}, we obtain
		\begin{align*}
			\Big|\int_{\X} a(x)(f(x)-f_B)dx\Big|&\leq \|a(f-f_B)\|_{L^1(B)}\\
			&\leq \|a\|_{L^q(B)}\|f-f_B\|_{L^{q'}(B)}\\
			&\lesssim |B|^{\frac{1}{q}-1} |B|^{\frac{1}{q'}} \frac{1}{\log\Big(e+\frac{x_B^n}{r_B}\Big)}\|f\|_{\mathrm{BMO}^{\log}(\X)}\\
			&\lesssim \|f\|_{\mathrm{BMO}^{\log}(\X)}\frac{1}{\log\Big(e+\frac{x_B^n}{r_B}\Big)}.
		\end{align*}
		This, together with \eqref{eq-product-Ls} and $\operatorname{supp}(a(f-f_B))\subset B$, implies that $a(f-f_B)$ is a constant multiple of an $s$-atom of log-type associated with $B$. Hence, by Lemma~\ref{lem-log-molecule},
		\[
\|a(f-f_B)\|_{H^1_r(\X)}\lesssim \|f\|_{\mathrm{BMO}^{\log}(\X)}.
\]
		Thus,
		$$\|f\|_{\mathrm{BMO}(\X)}+ \sup_{\text{$(H^1_r,q)$-atoms $a$ associated with $B$}} \|a(f-f_{B})\|_{H^1_r(\X)}\lesssim \|f\|_{\mathrm{BMO}^{\log}(\X)}.$$
		
		Conversely, let $f\in \mathrm{BMO}(\X)$ such that 
		\[
		\sup_{\text{$(H^1_r,q)$-atoms $a$ associated with $B$}} \|a(f-f_{B})\|_{H^1_r(\X)}<\vc.
		\]
		It suffices to prove that
		\begin{equation}\label{eq-bmolog-converse}
			 \log\Big(e+\frac{x_{\widetilde{B}}^n}{r_{\widetilde{B}}}\Big) \mathrm{MO}(f, \widetilde{B} )\lesssim \|f\|_{\mathrm{BMO}(\X)}+ \sup_{\text{$(H^1_r,q)$-atoms $a$ associated with $B$}} \|a(f-f_{B})\|_{H^1_r(\X)}
		\end{equation}
		for all balls $\widetilde{B}$. 
		
		If  $r_{\widetilde{B}}\geq x_{\widetilde{B}}^n$, then
		\[
		\log\Big(e+\frac{x_{\widetilde{B}}^n}{r_{\widetilde{B}}}\Big) \mathrm{MO}(f, \widetilde{B} )\simeq \mathrm{MO}(f, \widetilde{B} )\lesssim \|f\|_{\mathrm{BMO}(\X)}.
		\]
		It remains to consider the case  $0<r_{\widetilde{B}}<x_{\widetilde{B}}^n$. To do this, denote by $1_{\widetilde{B}}$ the characteristic function of $\widetilde{B}$. Let $g_{x_{\widetilde{B}}}:\X\to \R$ be as in \eqref{eq-gx0-definition}, and set $a_0=\frac{1}{2|\widetilde{B}|}(h-h_{\widetilde{B}})1_{\widetilde{B}}$, where $h=\sign (f-f_{\widetilde{B}})$. Then $a_0$ is an $(H^1_r,\infty)$-atom associated with the ball $\widetilde{B}$; in particular, it is also an $(H^1_r,q)$-atom for every $q\in(1,\infty)$. Moreover, by Lemma \ref{lem-gx0}, we obtain 
		\begin{align*}
			\log\Big(e+&\frac{x_{\widetilde{B}}^n}{r_{\widetilde{B}}}\Big)\mathrm{MO}(f, \widetilde{B})=2 \log\Big(e+\frac{x_{\widetilde{B}}^n}{r_{\widetilde{B}}}\Big) \int_{\widetilde{B}} a_0(x) (f(x)-f_{\widetilde{B}}) dx\\
			&\leq 4 (g_{x_{\widetilde{B}}})_{\widetilde{B}} \int_{\X} a_0(x) (f(x)-f_{\widetilde{B}}) dx\\
			&\lesssim \int_{\widetilde{B}} |a_0(x)| |f(x)-f_{\widetilde{B}}| |g_{x_{\widetilde{B}}}(x)- (g_{x_{\widetilde{B}}})_{\widetilde{B}}| dx + \Big|\int_{\widetilde{B}} a_0(x) (f(x)-f_{\widetilde{B}}) g_{x_{\widetilde{B}}}(x) dx\Big|\\
			&=:E+F.
		\end{align*}
		Using H\"older's inequality, \eqref{eq-classical-JN} and Lemma \ref{lem-gx0}, 
		\[
		\begin{aligned}
			E&\lesssim \|a_0\|_{L^\vc(\widetilde B)}\|f-f_{\widetilde{B}}\|_{L^2(\widetilde B)}\|g_{x_{\widetilde{B}}}- (g_{x_{\widetilde{B}}})_{\widetilde{B}}\|_{L^2(\widetilde B)}\\
			&\lesssim \|f\|_{\mathrm{BMO}(\X)}\|g_{x_{\widetilde{B}}}\|_{\mathrm{BMO}(\X)}\\
			&\lesssim \|f\|_{\mathrm{BMO}(\X)}\|g_{x_{\widetilde{B}}}\|_{\mathrm{BMO}_z(\X)}\\
			&\lesssim \|f\|_{\mathrm{BMO}(\X)}.
		\end{aligned}
		\]
	By the duality result and Lemma  \ref{lem-gx0}, 	
		\[
		 \begin{aligned}
		 		F &\le  \|a_0(f-f_{\widetilde{B}})\|_{H^1_r} \|g_{x_{\widetilde{B}}}\|_{\mathrm{BMO}_z}\\
		&\lesssim   \sup_{\text{$(H^1_r,q)$-atom $a$ associated with $B$}} \|a(f-f_{B})\|_{H^1_r}.
	\end{aligned} 
		\]
		This completes the proof.
		
	\end{proof}

	\subsection{Basic properties of the Riesz transform}
In this section, we establish the boundedness of the Riesz transform $\mathfrak R=\nabla(-\Delta_D)^{-1/2}$ on $H^1_r(\mathbb R^n_+)$ and recall its Riesz-transform characterization.

	We first establish the kernel estimates needed below.
	\begin{lemma}
		\label{lem-heat-kernel-estimates}
		For all $x,y \in \X$ and $t>0$, we have
		
		\begin{enumerate}[\rm (a)]
			\item $\displaystyle p_{t}(x,y)
			\lesssim \frac{1}{t^{n/2}}
			e^{-\frac{|x-y|^2}{ct}}\Big(1+\f{\sqrt t}{ x^n }+\f{\sqrt t}{ y^n }\Big)^{-1}$;
			
			\medskip
			
			\item $\displaystyle |\nabla_x p_{t}(x,y)|\lesssim \f{1}{t^{(n+1)/2}}e^{-\f{|x-y|^2}{ct}}\Big(1+\f{\sqrt t}{y^n}\Big)^{-1}$;
			
			\medskip
			
			\item $\displaystyle |\nabla_x\nabla_y p_{t}(x,y)|
			\lesssim \frac{1}{t^{(n+2)/2}}
			e^{-\frac{|x-y|^2}{ct}}$.
			\end{enumerate}
		\end{lemma}
		\begin{proof} 
			The estimate (c) is straightforward and hence we omit the details. 
			
			For the estimates (a) and (b), we first rewrite the heat kernel as
			\[
			p_{t}(x,y)
			= \frac{1}{(4\pi t)^{n/2}}
			e^{-\frac{|x-y|^2}{4t}}
			\Big(
			1-
			e^{-\frac{ x^n y^n}{t}}
			\Big).
			\]
			Obviously,
			\[
			p_{t}(x,y)\le \frac{1}{(4\pi t)^{n/2}}
			e^{-\frac{|x-y|^2}{4t}}.
			\]
			Hence, we need only to prove (a) for the case $\min\{x^n,y^n\}<\sqrt t$. Without loss of generality, we may assume that $\min\{x^n,y^n\}=x^n<\sqrt t$. 		Using the inequality $1-
			e^{-\frac{ x^n y^n}{t}}\lesssim \frac{ x^n y^n}{t}$, we further obtain  
			\[
			\begin{aligned}
				p_{t}(x,y)&\lesssim \frac{1}{(4\pi t)^{n/2}}
				e^{-\frac{|x-y|^2}{4t}}\frac{ x^n y^n}{t}\\
				&\lesssim \frac{1}{(4\pi t)^{n/2}}
				e^{-\frac{|x-y|^2}{4t}}\frac{ x^n |y^n-x^n|}{t} +\frac{1}{(4\pi t)^{n/2}}
				e^{-\frac{|x-y|^2}{4t}}\frac{ (x^n)^2 }{t}\\
				&\lesssim \frac{1}{(4\pi t)^{n/2}}
				e^{-\frac{|x-y|^2}{4t}}\frac{ x^n |x-y|}{t} +\frac{1}{(4\pi t)^{n/2}}
				e^{-\frac{|x-y|^2}{4t}}\frac{ x^n }{\sqrt t}\\
				&\lesssim  
				\frac{1}{(4\pi t)^{n/2}}e^{-\frac{|x-y|^2}{8t}}\frac{ x^n }{\sqrt t}.
			\end{aligned}
			\]
			This completes the proof of (a).
			
			The proof of (b) is similar to that of (a) by using the product rule and we omit the details.
			
			This completes the proof.
		\end{proof}
			\begin{lemma}\label{lem-Riesz-kernel}
			Let $\mathfrak{R}(x,y)$ be the associated kernel of $\mathfrak{R}$. Then we have
			\[
			|\mathfrak{R}(x,y)|\lesssim \f{1}{|x-y|^n}   \f{y^n}{|x-y|}, \ \ x\ne y
			\]
			and
			\[
			|\nabla_y\mathfrak{R}(x,y)|\lesssim \f{1}{|x-y|^{n+1}}, \ \ x\ne y.
			\]
		\end{lemma}
		\begin{proof}
			
			Using Lemma \ref{lem-heat-kernel-estimates}(b) and the subordination formula
			\begin{equation}\label{eq-subordination}
			(-\Delta_D)^{-1/2} = c\int_0^\vc \sqrt t e^{t\Delta_D} \f{dt}{t},
			\end{equation} we have
			\[
			\begin{aligned}
				|\mathfrak{R}(x,y)|&=c\Big|\int_0^\vc \sqrt t \nabla_x p_t(x,y)\f{dt}{t}\Big|\\
				&\lesssim \int_0^\vc \f{1}{t^{n/2}}e^{-\f{|x-y|^2}{ct}} \f{y^n}{\sqrt t} \f{dt}{t}\\
				&\lesssim \f{1}{|x-y|^n} \f{y^n}{|x-y|} .
			\end{aligned}
			\]

			The second estimate follows from Lemma \ref{lem-heat-kernel-estimates}(c) and the expression
			\[
			\mathfrak{R}(x,y) =c \int_0^\vc \sqrt t \nabla_x p_t(x,y)\f{dt}{t}.
			\]
			This completes the proof.
		\end{proof}

	\begin{proposition}\label{prop-R-Hr-bounded}
		The Riesz transform $\mathfrak R$ is bounded on $H^1_r(\X)$.
	\end{proposition}
		\begin{proof}
		Let $f\in H^1_r(\R^n_+)$. Denote by $\widetilde f$ the odd extension of $f$ to $\mathbb R^n$, i.e.,
		\[
		\widetilde f(x',x^n) = \begin{cases}
			f(x',x^n), & x^n>0,\\
			-f(x',-x^n), & x^n<0. 
		\end{cases}
		\]
		Then $\widetilde f\in H^1(\R^n)$ with $\|\widetilde f\|_{H^1(\R^n)}\lesssim \|f\|_{H^1_r(\R^n_+)}$. Moreover, 
		\begin{equation}\label{eq-odd-extension-riesz}
		\nabla (-\Delta_D)^{-1/2} f(x)=\nabla(-\Delta)^{-1/2}\widetilde f(x), \qquad x\in\R^n_+.
		\end{equation}
		Since the classical Riesz transform $\nabla (-\Delta)^{-1/2}$ is bounded on the Hardy space $H^1(\mathbb R^n)$, $\nabla (-\Delta)^{-1/2}\widetilde f\in H^1(\R^n)$. This, together with \eqref{eq-odd-extension-riesz}, implies $\nabla (-\Delta_D)^{-1/2} f\in H^1_r(\R^n_+)$; moreover,
		\[
		\|\nabla (-\Delta_D)^{-1/2}  f\|_{H^1_r(\R^n_+)} \le \|\nabla (-\Delta)^{-1/2}\widetilde f\|_{H^1(\R^n)}\lesssim \|\widetilde f\|_{H^1(\R^n)}\lesssim \|f\|_{H^1_r(\R^n_+)}.
		\] 
		
		This completes the proof.
	\end{proof}

	\begin{proposition}\label{prop-riesz-characterization}
		There exists a constant $C>1$ such that 
		\[
C^{-1}\|f\|_{H^1_r}
\leq \|f\|_{L^1}+\sum_{j=1}^{n}\|\mathfrak R_j(f)\|_{L^1}
\leq C\|f\|_{H^1_r},
\qquad f\in H^1_r(\X).
\]
	\end{proposition}

	\begin{proof} Let $f\in H^1_r(\R^n_+)$. From Proposition \ref{prop-R-Hr-bounded} and the fact $H^1_r(\R^n_+)\subset L^1(\R^n_+)$,
		\[
		\|f\|_{L^1}+\sum_{j=1}^{n} \|\mathfrak R_j(f)\|_{L^1}\leq C \|f\|_{H^1_r}.
		\]
		It suffices to prove 
		\[
		\|f\|_{H^1_r}\lesssim \|f\|_{L^1}+\sum_{j=1}^{n} \|\mathfrak R_j(f)\|_{L^1}, \ \ f\in H^1_r(\R_+^n).
		\]
		Denote by $\widetilde f$ the odd extension of $f$ to $\mathbb R^n$ as in the proof of Proposition \ref{prop-R-Hr-bounded}. Then we have
		\[
		\|\nabla (-\Delta_D)^{-1/2} f\|_{L^1}+\| f\|_{L^1} \simeq  \|\nabla (-\Delta)^{-1/2}\widetilde f\|_{L^1}+\|\widetilde f\|_{L^1}\gtrsim \|\widetilde f\|_{H^1}\ge \| f\|_{H^1_r}. 
		\]
		This completes the proof.
	\end{proof}
	
We conclude this section by proving the following estimates.
	
	\begin{lemma}\label{lem-Ra-estimates}
		Let $1<q, s<\infty$. Then there exists a constant $C>0$ such that 
		\begin{enumerate}[\rm (i)]
			\item\label{item-Ra-offdiagonal} for all $(H^1_r,q)$-atoms $a$ associated with balls $B$, and all  $j\ge 0$,
			\[\|\mathfrak{R}a\|_{L^q(S_j(B))}\leq C 2^{-j} |2^jB|^{\frac{1}{q}-1}.\]
			
			\item\label{item-bRa-L1} for all $(H^1_r,s)$-atoms $a$ associated with balls $B$, and all $b\in \mathrm{BMO}(\X)$,
			\[\|(b-b_B)\mathfrak{R}a\|_{L^1}\leq C \|b\|_{\mathrm{BMO}}.\]
		\end{enumerate}
	\end{lemma}

	\begin{proof}	
		\eqref{item-Ra-offdiagonal} The case $j=0,1$ is trivial since $\mathfrak R$ is bounded on $L^q(\X)$. For $j\ge 2$, we consider the following two cases:
		
		\noindent {\sl Case 1:} $0<r_B\leq\frac{x_B^n}{4}$. Then, $\displaystyle \int_{\X} a(y)dy=0$. Thus, by H\"older's inequality, Lemma \ref{lem-Riesz-kernel} and the fact $\|a\|_{L^1}\le 1$, we have
		\begin{align}\label{eq-Ra-cancellative-atom}
			\|\mathfrak Ra\|_{L^q(S_j(B))}
			&\leq \int_B |a(y)|\Big(\int_{S_j(B)} |\mathfrak R(x,y)-\mathfrak R(x,x_B)|^q\,dx\Big)^{1/q}dy\nonumber\\
			&\lesssim \int_B |a(y)|\Big(\int_{S_j(B)} \Big[\f{|y-x_B|}{|x-y|^{n+1}}\Big]^q dx\Big)^{1/q}dy\nonumber\\
			&\lesssim \|a\|_{L^1}\Big(\int_{S_j(B)} \Big[\f{r_B}{(2^jr_B)^{n+1}}\Big]^q dx\Big)^{1/q}\nonumber\\
			&\lesssim 2^{-j} |2^j B|^{\frac{1}{q}-1}.
		\end{align}
		
		{\sl Case 2:} $\frac{x_B^n}{4}<r_B<\frac{x_B^n}{2}$. Then, by H\"older's inequality, Lemma \ref{lem-Riesz-kernel} and the fact $y^n\simeq r_B$ for $y\in B$,
		\begin{align}\label{eq-Ra-boundary-atom}
			\|\mathfrak R a\|_{L^q(S_j(B))}
			&\leq \int_B |a(y)|\Big(\int_{S_j(B)} |\mathfrak R(x,y)|^q\,dx\Big)^{1/q}dy\nonumber\\
			&\lesssim \int_B |a(y)|\Big(\int_{S_j(B)} \f{1}{|x-y|^{qn}}\Big[\f{r_B}{|x-y|}\Big]^q dx\Big)^{1/q}dy\nonumber\\
			&\lesssim 2^{-j}|2^jB|^{\frac1q-1}.
			\end{align}
		
		Taking \eqref{eq-Ra-cancellative-atom} and \eqref{eq-Ra-boundary-atom} into account, we complete the proof of (i).
		
		\medskip

		\eqref{item-bRa-L1} Let $q:=\frac{s+1}{2}\in(1,s)$. An $(H^1_r,s)$-atom is also an $(H^1_r,q)$-atom. Hence, by H\"older's inequality, \eqref{eq-classical-JN}, \eqref{eq-BMO-dilated-averages}, and part~(i),
		\begin{align*}
			&\|(b-b_B)\mathfrak Ra\|_{L^1}=  \sum_{j=0}^\infty \|(b-b_B)\mathfrak Ra\|_{L^1(S_j(B))}\\
			&\hskip2cm\leq \sum_{j=0}^\infty \|b-b_B\|_{L^{q'}(S_j(B))}\|\mathfrak Ra\|_{L^q(S_j(B))}\\
			&\hskip2cm\lesssim \sum_{j=0}^\infty |2^{j+1} B|^{\frac{1}{q'}}(j+2) \|b\|_{\mathrm{BMO}}\times  2^{-j}|2^j B|^{\frac{1}{q}-1}\\
			&\hskip2cm\lesssim  \|b\|_{\mathrm{BMO}}.
		\end{align*}
		This completes the proof.
	\end{proof}

	\section{Proofs of the main results}\label{sec-main-proof}

	We now give the proof of Theorem \ref{mainthm}.

	\begin{proof}[Proof of Theorem \ref{mainthm}]
(a) It is enough to prove the estimate for finite linear combinations of $(H^1_r,q)$-atoms, which form a dense subspace of $H^1_r(\X)$. Let
\[
f=\sum_{j=1}^{N}\lambda_j a_j,
\qquad
\sum_{j=1}^{N}|\lambda_j|\leq 2\|f\|_{H^1_r},
\]
where each $a_j$ is an $(H^1_r,q)$-atom associated with a ball $B_j$. Therefore,
		\begin{align*}
			\Big|[b,\mathfrak R](f)(x)\Big|&= \Big|\mathfrak R\Big(\Big(b(x)-b(\cdot)\Big)\sum_{j=1}^{N} \lambda_j a_j(\cdot)\Big)(x)\Big|\\
			&\leq \sum_{j=1}^{N} |\lambda_j|\,|b(x)-b_{B_j}|\,|\mathfrak R(a_j)(x)|
			+\Big|\mathfrak R\Big(\sum_{j=1}^{N} \lambda_j (b_{B_j}-b(\cdot))a_j(\cdot)\Big)(x)\Big|.
		\end{align*}
		This, together with the fact that $\mathfrak R$ is of weak type $(1,1)$ (since $\mathfrak{R}$ is a Calder\'on-Zygmund operator), implies
			\begin{align*}
			\|[b,\mathfrak R](f)\|_{L^{1,\infty}}&\lesssim \sum_{j=1}^{N} |\lambda_j| \|(b-b_{B_j})\mathfrak R(a_j)\|_{L^1} + \|\mathfrak R\|_{L^1\to L^{1,\infty}}\sum_{j=1}^{N} |\lambda_j|\|(b-b_{B_j})a_j\|_{L^1}.
		\end{align*}
By Lemma \ref{lem-Ra-estimates},
\[
\begin{aligned}
	\sum_{j=1}^{N} |\lambda_j| \|(b-b_{B_j})\mathfrak R(a_j)\|_{L^1}&\lesssim \|b\|_{\mathrm{BMO}(\X)}\sum_{j=1}^{N} |\lambda_j|\\
	&\lesssim  \|b\|_{\mathrm{BMO}(\X)} \|f\|_{H^1_r}.
\end{aligned}
\]
In addition, using H\"older's inequality and \eqref{eq-classical-JN}, we have
\[
\begin{aligned}
	\|(b-b_{B_j})a_j\|_{L^1}&\le \|b-b_{B_j}\|_{L^{q'}(B_j)}\|a_j\|_{L^q}\\
	&\lesssim \|b\|_{\mathrm{BMO}(\X)},
\end{aligned}
\]
which implies
\[
\begin{aligned}
	\sum_{j=1}^{N} |\lambda_j|\|(b-b_{B_j})a_j\|_{L^1}&\lesssim \|b\|_{\mathrm{BMO}(\X)}\sum_{j=1}^{N} |\lambda_j|\\
	&\lesssim  \|b\|_{\mathrm{BMO}(\X)} \|f\|_{H^1_r}.
\end{aligned}
\]
Consequently,
\[
\|[b,\mathfrak R](f)\|_{L^{1,\infty}}\lesssim \|b\|_{\mathrm{BMO}(\X)} \|f\|_{H^1_r}.
\]		
The bound is independent of $N$, so density yields the asserted extension to all of $H^1_r(\X)$. This completes the proof of (a). 

\bigskip

(b) Assume first that $b\in\mathrm{BMO}^{\log}(\X)$. Fix $q\in(1,\infty)$. Since $\mathrm{BMO}^{\log}(\X)\subset\mathrm{BMO}(\X)$, the commutator is bounded on $L^2(\X)$; hence the standard atomic extension criterion reduces the proof to showing that
		\begin{equation}\label{eq-commutator-atom-bound}
			\|[b,\mathfrak{R}](a)\|_{H^1_r}\lesssim \|b\|_{\mathrm{BMO}^{\log}(\X)}
		\end{equation}
		for all $(H^1_r,q)$-atoms $a$ associated with balls $B$. 
		
		Let $p:=\frac{q+1}{2}\in(1,q)$ and choose $s\in(1,\infty)$ so that $\frac{1}{p}=\frac{1}{s}+\frac{1}{q}$. For any $j\ge 0$, by H\"older's inequality,  Lemma \ref{lem-Ra-estimates}, \eqref{eq-classical-JN} and \eqref{eq-BMO-dilated-averages}, we obtain
		\begin{align}\label{eq-bRa-molecule-size}
			&\|(b-b_B)\mathfrak{R}a\|_{L^p(S_j(B))}\leq \|b-b_B\|_{L^s(S_j(B))} \|\mathfrak{R}a\|_{L^q(S_j(B))}\nonumber\\
			&\hskip1cm\lesssim |2^{j+1} B|^{\frac{1}{s}}(j+1)  \|b\|_{\mathrm{BMO}(\X)}\times  2^{-j}|2^j B|^{\frac{1}{q}-1}\nonumber\\
			&\hskip1cm\lesssim 2^{-j/2} |2^j B|^{\frac{1}{p}-1} \|b\|_{\mathrm{BMO}^{\log}(\X)},
		\end{align}
		where we used $(j+1)2^{-j}\lesssim 2^{-j/2}$. 
		
		On the other hand, by H\"older's inequality, Lemma~\ref{lem-Ra-estimates}\eqref{item-Ra-offdiagonal}, and Lemma~\ref{lem-bmolog-dilated-ball}, we obtain
		\begin{align*}
			&\Big|\int_{\X} (b(x)-b_B)\mathfrak{R}a(x)dx\Big|\leq \sum_{j=0}^\infty \int_{S_j(B)} |b(x)-b_B| |\mathfrak{R}a(x)|dx\\
			&\hskip1cm\leq \sum_{j=0}^\infty \|b-b_B\|_{L^{q'}(S_j(B))} \|\mathfrak{R}a\|_{L^q(S_j(B))}\\
			&\hskip1cm\lesssim \sum_{j=0}^\infty  |2^{j+1} B|^{\frac{1}{q'}}(j+2) \frac{1}{\log\Big(e+\frac{x_B^n}{2^{j+1} r_B}\Big)} \|b\|_{\mathrm{BMO}^{\log}(\X)}\times  2^{-j}|2^j B|^{\frac{1}{q}-1}\\
			&\hskip1cm\lesssim  \frac{\|b\|_{\mathrm{BMO}^{\log}(\X)}}{\log\Big(e+\frac{x_B^n}{r_B}\Big)},
		\end{align*}
		where in the last inequality we used the following inequality
		\[
		\sum_{j=0}^\infty(j+2)2^{-j}\frac{1}{\log\Big(e+\frac{x_B^n}{2^{j+1} r_B}\Big)}\lesssim \frac{1}{\log\Big(e+\frac{x_B^n}{r_B}\Big)}.
		\]
This elementary bound follows, for example, by setting $A=x_B^n/r_B$ and splitting the sum at $j=\lfloor \frac12\log_2 A\rfloor$ (the case of bounded $A$ being immediate).
		Together with \eqref{eq-bRa-molecule-size}, this shows that $(b-b_B)\mathfrak R a$ is a constant multiple of a $(p,1/2)$-molecule of log-type. Hence Lemma~\ref{lem-log-molecule} gives
		\[
		\|(b-b_B)\mathfrak R a\|_{H^1_r}\lesssim \|b\|_{\mathrm{BMO}^{\log}(\X)}.
		\]
		Thus, by Lemma \ref{lem-bmolog-characterization} and the $H^1_r$-boundedness of $\mathfrak{R}$ (Proposition \ref{prop-R-Hr-bounded}),
		\begin{align*}
			\|[b,\mathfrak{R}](a)\|_{H^1_r}&\leq \|(b-b_B)\mathfrak{R}a\|_{H^1_r}+ \|\mathfrak{R}(a(b-b_B))\|_{H^1_r}\\
			&\lesssim \|b\|_{\mathrm{BMO}^{\log}(\X)} + \|\mathfrak{R}\|_{H^1_r\to H^1_r} \|a(b-b_B)\|_{H^1_r} \lesssim \|b\|_{\mathrm{BMO}^{\log}(\X)}.
		\end{align*}
		This proves \eqref{eq-commutator-atom-bound}, and thus completes the proof.
		
		\medskip

		Conversely, assume that $b\in\mathrm{BMO}(\X)$ and that $[b,\mathfrak R]$ is bounded on $H^1_r(\X)$. Since $H^1_r(\X)$ admits a characterization in terms of the singular integral operators $ \mathfrak{R}_1,\ldots, \mathfrak{R}_n$, we have
		\[H^1_r(\X)=\Big\{f\in L^1(\X): \mathfrak{R}_1(f),\mathfrak{R}_2(f),\ldots,\mathfrak{R}_n(f)\in L^1(\X)\Big\}\]
		and
		\begin{equation}\label{eq-riesz-Hr-characterization}
			\|f\|_{H^1_r}\simeq \|f\|_{L^1}+\sum_{j=1}^{n} \|\mathfrak{R}_j(f)\|_{L^1}\quad\text{for all } f\in H^1_r(\X).
		\end{equation}   
		
		For every $1\leq j\leq n$ and every $(H^1_r,s)$-atom $a$ associated with a ball $B$, it follows from Lemma \ref{lem-Ra-estimates}\eqref{item-bRa-L1} and the fact that the commutator $[b,\mathfrak R_j]$ is bounded on $H^1_r(\X)$  that
		\begin{align}\label{eq-necessity-riesz-component}
			\Big\|\mathfrak{R}_j\Big((b-b_B)a\Big)\Big\|_{L^1}
			&\leq \|(b-b_B)\mathfrak{R}_j(a)\|_{L^1}
			+ \|[b,\mathfrak R_j](a)\|_{L^1}\nonumber\\
			&\lesssim \|b\|_{\mathrm{BMO}(\X)}
			+ \|[b,\mathfrak{R}_j]\|_{H^1_r\to L^1}\|a\|_{H^1_r}\nonumber\\
			&\lesssim \|b\|_{\mathrm{BMO}(\X)}
			+ \|[b,\mathfrak{R}_j]\|_{H^1_r\to H^1_r}.
		\end{align}
	Moreover, by H\"older's inequality and the John--Nirenberg inequality, we have 
	\begin{equation}\label{eq-product-L1}
		\|(b-b_B)a\|_{L^1}\lesssim \|b\|_{\mathrm{BMO}(\X)}.
	\end{equation}

	Combining \eqref{eq-riesz-Hr-characterization}, \eqref{eq-necessity-riesz-component}, \eqref{eq-product-L1}, we obtain
	\[
	\begin{aligned}
		\|(b-b_B)a\|_{H^1_r}&\simeq  \|(b-b_B)a\|_{L^1} +\sum_{j=1}^n\Big\|\mathfrak{R}_j\Big((b-b_B)a\Big)\Big\|_{L^1}\\
		&\lesssim  \|b\|_{\mathrm{BMO}(\X)} +\sum_{j=1}^n\|[b,\mathfrak{R}_j]\|_{H^1_r\to H^1_r}.
	\end{aligned}
	\]
	
	Applying Lemma~\ref{lem-bmolog-characterization} with atom exponent $s$ gives the necessity estimate
\[
\|b\|_{\mathrm{BMO}^{\log}(\X)}
\lesssim \|b\|_{\mathrm{BMO}(\X)}+\sum_{j=1}^{n}\|[b,\mathfrak R_j]\|_{H^1_r\to H^1_r}.
\]
Together with the sufficiency estimate proved above and the embedding
$\mathrm{BMO}^{\log}(\X)\subset\mathrm{BMO}(\X)$, this yields the quantitative norm equivalence in Theorem~\ref{mainthm}.
		This completes the proof.
		
	\end{proof}

\section{Appendix}
This section is devoted to the proof of Lemma \ref{lem-gx0}.
\begin{proof}[Proof of Lemma  \ref{lem-gx0}.] Recall that 
		\begin{equation*} 
		g_{x_0}(x)=\max\Big\{0, 1+ \log\frac{x_0^n}{|x-x_0|}\Big\}.
	\end{equation*}
	It suffices to prove that 
\begin{equation}\label{eq-gx0-bmo}
\|g_{x_0}\|_{\mathrm{BMO}(\X)} \le C
\end{equation}
and
\begin{equation}\label{eq-gx0-boundary-average}
\sup_{\substack{B:\,\text{ball}\\ r_B\ge x_B^n/4}}\fint_B |g_{x_0}| \le C,
\end{equation}
where the constant $C$ is independent  of $x_0$.

To prove \eqref{eq-gx0-bmo}, note that the positive-part map is $1$-Lipschitz. Thus it suffices to show, uniformly in $x_0$, that $\log|\cdot-x_0|\in\mathrm{BMO}(\X)$. For this, it is enough to prove
\[
\sup_{B \subset \X} \inf_{c\in \R} \fint_B |\log |y-x_0| - c|\,dy \lesssim 1,
\]
where the supremum is taken over all balls $B\subset\X$.

It suffices to show that
\begin{equation} 
	\fint_B |\log |y-x_0| - c_B| \, dy \lesssim 1,
\end{equation}
for all balls $B\subset \X$ with 
\[
c_B = 
\begin{cases}
	\log |x_B-x_0|, & |x_B-x_0| \ge 2r_B, \\
	\log r_B, & |x_B-x_0| < 2r_B.
\end{cases}
\]  
If $|x_B-x_0| \ge 2r_B$, then $|x_B-x_0|/2\le |y-x_0|\le 2|x_B-x_0|$. Consequently,  
\[
\begin{aligned}
	\fint_B |\log |y-x_0| - c_B| dy &=\fint_B \Big|\log \Big(\f{|y-x_0|}{|x_B-x_0|}\Big)\Big|dy\\
	&\lesssim 1.
\end{aligned}
\]
If $|x_B-x_0| < 2r_B$, then $|y-x_0|\le 3r_B$ and $B\subset B(x_0,4r_B)$. Therefore,  
\[
\begin{aligned}
	\fint_B |\log |y-x_0| - c_B|\,dy
&\lesssim \frac{1}{r_B^n}\int_{B(x_0,3r_B)}
\Big|\log \Big(\f{|y-x_0|}{r_B}\Big)\Big|\,dy\\
	&\lesssim 1.
\end{aligned}
\]
This ensures \eqref{eq-gx0-bmo}.

\medskip
We now prove \eqref{eq-gx0-boundary-average} directly. By scaling,
\begin{equation}\label{eq-meanvalue-B0}
\fint_{B(x_0,x_0^n)} g_{x_0}(y)\,dy\lesssim 1.
\end{equation}
Indeed, after the change of variables $y=x_0+x_0^n z$, the integral reduces to an absolute constant depending only on the dimension.

Let $B$ be a ball with $r_B\ge x_B^n/4$. We write
\[
\fint_{B} g_{x_0}(y)\, dy \le \f{1}{|B|}\int_{B\cap B(x_0,x_0^n/3)}  g_{x_0} (y)\, dy + \f{1}{|B|}\int_{B\setminus B(x_0,x_0^n/3)}  g_{x_0} (y)\, dy. 
\]
Since $g_{x_0}(y)\le \max\{0,1+ \log 3\}$ for $B\setminus B(x_0,x_0^n/3)$, we have
\[
\f{1}{|B|}\int_{B\setminus B(x_0,x_0^n/3)}  g_{x_0} (y)\, dy\lesssim 1.
\]

If $B\cap B(x_0,x_0^n/3)\ne \emptyset$, choose $y\in B\cap B(x_0,x_0^n/3)$. Since $x_B^n\le 4r_B$, we have
\[
x_0^n\le x_B^n+|x_B-y|+|y-x_0|\le 5r_B+\frac{x_0^n}{3},
\]
and hence $r_B\gtrsim x_0^n$. Consequently,
\[
\begin{aligned}
	\f{1}{|B|}\int_{B\cap B(x_0,x_0^n/3)}  g_{x_0} (y)\, dy&\le \f{|B(x_0,x_0^n/3)|}{|B|}\fint_{B(x_0,x_0^n/3)}  g_{x_0} (y)\, dy\\
	&\lesssim \fint_{B(x_0,x_0^n)}  g_{x_0} (y)\, dy\\
	&\lesssim 1.
\end{aligned}
\]
where in the last inequality we used \eqref{eq-meanvalue-B0}.

This completes the proof of \eqref{eq-gx0-boundary-average} and hence completes the proof of the lemma.
\end{proof}

\begin{acknowledgement}
	The Anh Bui and Xuan Thinh Duong were supported by Australian Research Council grant DP260101083. Luong Dang Ky was supported by the Vietnam Ministry of Education and Training (through the National Key Program for the Development of Mathematics in the 2021--2030 period) under Grant No. B2026-CTT-01.
\end{acknowledgement}


\begin{thebibliography}{99}

		
		\bibitem{BDK} T. A. Bui, X. T. Duong, and F. K. Ly, Maximal function characterizations for Hardy spaces on spaces of homogeneous type with finite measure and applications. J. Funct. Anal. 278 (2020), 108423.
		
		
		
		\bibitem{BLL}  T. A. Bui, J. Li and F. K. Ly, $T1$ criteria for generalised Calder\'on--Zygmund type operators on Hardy and $\mathrm{BMO}$ spaces associated to Schr\"odinger operators and applications. Ann. Sc. Norm. Super. Pisa Cl. Sci. (5) 18 (2018), no. 1, 203--239.
		
		\bibitem{CY} M. Cao and K. Yabuta, Spaces associated with Neumann Laplacian. J. Geom. Anal. 32 (2022), Article No. 59.
		

		\bibitem{CRW} R. R. Coifman, R. Rochberg, and G. Weiss,
		Factorization theorems for Hardy spaces in several variables.
		Ann. of Math. (2) 103 (1976), no. 3, 611--635.


		
		%
		%


		
		 
		\bibitem{Ky13} L. D. Ky, Bilinear decompositions and commutators of singular integral operators. Trans. Amer. Math. Soc. 365 (2013), no. 6, 2931--2958.		
		
		\bibitem{Ky15} L. D. Ky,  Endpoint estimates for commutators of singular integrals related to Schr\"odinger operators. Rev. Mat. Iberoam. 31 (2015), no. 4, 1333--1373.
		
		\bibitem{Ky25} L. D. Ky, Generalized Calder\'on--Zygmund operators on the Hardy space $H^1_r(\X)$. Banach J. Math. Anal. 19 (2025), no. 2, Paper No. 20.
		

		

		

		\bibitem{Mi} A. Miyachi, $H^p$ space over open subsets of $\mathbb R^n$. Studia Math. 95 (1990), 205--228.
		

		

		
		\bibitem{Pe} C. P\'erez, Endpoint estimates for commutators of singular integral operators. J. Funct. Anal. 128 (1995), no. 1, 163--185.
		

		

		
		\bibitem{YYZ} Da. Yang,  Do. Yang and  Y. Zhou,  Localized Morrey-Campanato spaces on metric measure spaces and applications to Schr\"odinger operators. Nagoya Math. J. 198 (2010), 77--119.
		
		\bibitem{YZ} D. Yang and Y. Zhou, Localized Hardy spaces $H^1$ related to admissible functions on RD-spaces and applications to Schr\"odinger operators. Trans. Amer. Math. Soc. 363 (2011), no. 3, 1197--1239.
		
		
	\end{thebibliography}
\end{document}